\documentclass[12pt]{amsart}
\usepackage[T1]{fontenc}
\usepackage{lmodern}
\usepackage[english]{babel}
\usepackage{amsmath}
\usepackage{amssymb}
\usepackage{amsxtra}
\usepackage{amscd}
\usepackage[margin=2cm]{geometry}
\usepackage{microtype}
\usepackage{csquotes}
\usepackage[
    backend=biber,
    style=numeric,
    sorting=nyt
]{biblatex}
\usepackage{xcolor}
\definecolor{linkblue}{HTML}{263F70}
\definecolor{citeviolet}{HTML}{563D70}
\definecolor{urlblue}{HTML}{345775}

\usepackage[
    colorlinks=true,
    linkcolor=linkblue,
    citecolor=citeviolet,
    urlcolor=urlblue,
    bookmarksnumbered=true,
    pdfpagemode=UseNone,
    final
]{hyperref}
\usepackage{tikz}
\newtheorem{theorem}{Theorem}[section]
\newtheorem{lemma}[theorem]{Lemma}
\newtheorem{proposition}[theorem]{Proposition}
\newtheorem{corollary}[theorem]{Corollary}
\newtheorem{maintheorem}{Theorem}

\theoremstyle{definition}
\newtheorem{definition}[theorem]{Definition}
\theoremstyle{remark}
\newtheorem{remark}[theorem]{Remark}
\newcommand{\mbb}{\mathbb}
\newcommand{\opn}{\operatorname}
\newcommand{\mcl}{\mathcal}
\newcommand{\Cx}{\mbb{C}}
\newcommand{\CG}{\mbb{C}[G]}
\DeclareMathOperator{\Der}{Der}
\DeclareMathOperator{\Inn}{Inn}
\DeclareMathOperator{\QInn}{QInn}
\title{Number of ends of conjugacy diagrams}
\author{Alexey V. Naianzin}
\email{naianzin.av@phystech.edu}
\author{Andronick A. Arutyunov}
\email{andronick.arutyunov@gmail.com}
\date{}
\begin{document}
\begin{abstract}
For each $n\geqslant2$, we construct a finitely generated group $W_n$ with a
conjugacy class such that the conjugacy diagram of this class has exactly
$n$ ends. For this class, we compute the quotient spaces of the space of
derivations of the group algebra $\Cx[W_n]$ modulo quasi-inner and inner
derivations, as well as the corresponding component of the first Hochschild
cohomology.
\end{abstract}

\maketitle

\section*{Introduction}
Conjugacy diagrams were introduced in \cite{aru23,Arutyunov2023} in the study of derivations of group algebras. Let $G$ be a finitely generated group. The space of derivations of the group algebra $\CG$ decomposes over the conjugacy classes $[u]$ of $G$, and for each class we have the inclusions
\[
\Inn(\Gamma_{[u]})\subseteq\QInn(\Gamma_{[u]})\subseteq\Der(\Gamma_{[u]})
\]
of the spaces of inner, quasi-inner, and all derivations corresponding to the class $[u]$ (for the definitions, see Subsection~\ref{sec:der}). If the class $[u]$ is infinite, then the dimension of $\QInn(\Gamma_{[u]})/\Inn(\Gamma_{[u]})$ is one less than the number of ends of the corresponding component of the conjugacy diagram \cite[Corollary~2]{aru23}; see Theorem~\ref{thm:qinn} below. In previously known examples, the number of ends is $0$, $1$, $2$, or infinite. However, examples of conjugacy diagrams with a finite number of ends greater than $2$ were not known to the authors. The following result fills this gap.

\begin{maintheorem}\label{thm:A}
For each $n\geqslant2$, the wreath product $W_n=\mbb Z\wr_{X_n}\mcl G_n$ of
the group $\mbb Z$ and the Houghton group $\mcl G_n$ with respect to the
natural action of $\mcl G_n$ on the set $X_n$ of $n$ rays is finitely
generated, and the component of the conjugacy diagram of $W_n$ containing
the element $u=(e_{x_0},1)$ has exactly $n$ ends.
\end{maintheorem}

Precise definitions are given in Subsections~\ref{sec:houghton}
and~\ref{sec:Wn-ends}. In particular, the dimension of
$\QInn(\Gamma_{[u]})/\Inn(\Gamma_{[u]})$ can take any finite value
(Corollary~\ref{cor:anydim}).

The applications to derivations and Hochschild cohomology are structured
as follows. The first Hochschild cohomology decomposes into a direct sum
over the conjugacy classes:
\[
\opn{HH}^1(\CG,\CG)\cong\bigoplus_{[u]}\opn{HH}^1_{[u]}(\CG,\CG),
\qquad
\opn{HH}^1_{[u]}(\CG,\CG)\cong\Der(\Gamma_{[u]})/\Inn(\Gamma_{[u]})
\]
(Subsection~\ref{sec:hochschild}), and the restriction of characters to the
loops at the object $u$ yields the exact sequence
\[
0\longrightarrow\QInn(\Gamma_{[u]})/\Inn(\Gamma_{[u]})
\longrightarrow\opn{HH}^1_{[u]}(\CG,\CG)
\longrightarrow\opn{Adm}_u(G)\longrightarrow0,
\]
where $\opn{Adm}_u(G)\subseteq\opn{Hom}(C_G(u),\Cx)$ is the space of
admissible characters of the centralizer
(Subsection~\ref{subsec:admissible}).
If $[u]$ is infinite, then the dimension of the first term equals
$e(\opn{conj}(G)_{[u]})-1$; thus this dimension is determined by the number
of ends of the component of the conjugacy diagram. To compute the last
term, one must determine which characters of the centralizer $C_G(u)$
extend to locally finite characters of the groupoid $\Gamma_{[u]}$. This
sequence is a fragment of the long exact cohomology sequence of the group
$G$ with coefficients in the permutation module $\opn{span}_{\Cx}[u]$
(Remark~\ref{rem:LES}). For the group $W_n$, both outer terms are computed
explicitly.

\begin{maintheorem}\label{thm:B}
For the class $[u]$ from Theorem~\ref{thm:A}, we have
\[
\dim\bigl(\QInn(\Gamma_{[u]})/\Inn(\Gamma_{[u]})\bigr)=n-1,
\qquad
\dim\bigl(\Der(\Gamma_{[u]})/\QInn(\Gamma_{[u]})\bigr)=
\begin{cases}
2,&n=2,\\
1,&n\geqslant3,
\end{cases}
\]
and, consequently,
\[
\dim\opn{HH}^1_{[u]}\bigl(\Cx[W_n],\Cx[W_n]\bigr)=
\begin{cases}
3,&n=2,\\
n,&n\geqslant3.
\end{cases}
\]
\end{maintheorem}

Theorem~\ref{thm:A} follows from Proposition~\ref{prop:Wn-fg} and
Theorem~\ref{thm:Wn-ends}; Theorem~\ref{thm:B} follows from
Theorems~\ref{thm:Wn-ends}, \ref{thm:qinn}, and~\ref{thm:Wn-admissible}
together with Corollary~\ref{cor:Wn-HH-component}.

The conjugacy diagram of the class $[u]$ is isomorphic to the Schreier
graph of the pair $(G,C_G(u))$ (Proposition~\ref{prop:conj-schreier}), and
the number of ends of this diagram equals the number of ends of this pair
in the sense of \cite{scott1977ends}.
Pairs with an arbitrary finite number of ends are known: Houghton groups
provide such pairs \cite[Example~2.1]{scott1977ends}; see
also~\cite[\S~6]{azu24}.
Theorem~\ref{thm:A} gives such pairs in which the subgroup is the
centralizer of an element: in $W_n$, the centralizer of $u$ equals
$A_n\rtimes H_n$, and the stabilizer $H_n$ lifts to a centralizer
(Proposition~\ref{prop:centralizer}). Thus, for each $n\geqslant2$,
Theorem~\ref{thm:A} gives an example of a pair with $n$ ends consisting of
a group and the centralizer of an element of this group.
The number of ends of pairs and related invariants were recently studied
in~\cite{genevois2026}; subgroups with narrow Schreier graphs are considered
in~\cite{azu24}. The first Hochschild cohomology of group algebras of finite
groups is actively studied in the modular case (see, e.g.,
\cite{ABRyD2026}).

This paper consists of two sections. In Section~\ref{sec:ends}, we compute
the number of ends: Subsection~\ref{sec:conj-schreier} recalls the relation
between conjugacy diagrams and Schreier graphs,
Subsection~\ref{sec:houghton} introduces Houghton groups, and in
Subsection~\ref{sec:Wn-ends} we construct the group $W_n$. In
Section~\ref{sec:applications}, we describe the components of the space of
derivations (Subsection~\ref{sec:der}) and admissible characters
(Subsection~\ref{subsec:admissible}); admissible characters of $W_n$ are
computed in Subsection~\ref{subsec:Wn-der}. In
Subsection~\ref{sec:hochschild}, these results are stated in terms of
Hochschild cohomology.

\section{Computing the number of ends}\label{sec:ends}

\subsection{Conjugacy diagrams and Schreier graphs}\label{sec:conj-schreier}
\begin{definition}\label{def:conj-graph}
Suppose $G$ is a group generated by a finite set $S$.
The \emph{conjugacy diagram} $\opn{conj}(G)$ is the undirected graph with
vertex set $G$ in which, for all $g\in G$ and $s\in S$, the vertices $g$
and $s^{-1}gs$ are joined by an edge with label $s$; loops are allowed.
The connected component containing $u$ is denoted by
$\opn{conj}(G)_{[u]}$ and is called the
\emph{conjugacy diagram of the class} $[u]$.
\end{definition}

In \cite[Definition~1]{aru23}, the diagram $\opn{sk}(G;\lambda)$ is
associated with the conjugation action $\lambda(h,x)=hxh^{-1}$. Our diagram
$\opn{conj}(G)$ is obtained from $\opn{sk}(G;\lambda)$ by forgetting the
orientation and identifying edges with the same endpoints and labels.

\begin{definition}\label{def:ends}
Let $Y$ be a connected locally finite graph.
The \emph{number of ends} of $Y$ is the supremum of the number of infinite
connected components of $Y\setminus K$ over all finite sets $K$ of
vertices. We denote this number by
$e(Y)\in\mbb Z_{\geqslant0}\cup\{\infty\}$.
\end{definition}

\begin{definition}\label{def:schreier}
Suppose $G$ acts on a set $X$ from the right; we denote the action by
$x\cdot g$. The \emph{Schreier graph} of this action with respect to $S$
is the undirected graph with vertex set $X$ in which $x$ and $x\cdot s$
are joined by an edge with label $s$ for all $x\in X$ and $s\in S$.

If $H\leqslant G$, then the Schreier graph of the natural right action of
$G$ on the set of cosets $H\backslash G$ is denoted by
$\opn{Sch}(G,H,S)$. The number of ends of the pair $(G,H)$ is the number of
ends of $\opn{Sch}(G,H,S)$; this number does not depend on the choice of
the finite generating set \cite[Lemma~1.1]{scott1977ends}.
\end{definition}

We use the convention $u^g=g^{-1}ug$.
\begin{proposition}\label{prop:conj-schreier}
For each $u\in G$, the map
\[
u^g\longmapsto C_G(u)g
\]
defines an isomorphism
\[
\opn{conj}(G)_{[u]}\cong
\opn{Sch}\bigl(G,C_G(u),S\bigr).
\]
\end{proposition}

\begin{proof}
The equality $u^g=u^h$ is equivalent to $hg^{-1}\in C_G(u)$, i.e., to
$C_G(u)g=C_G(u)h$. Therefore the map is well defined and bijective. In
addition, $(u^g)^s=u^{gs}$, while the edge with label $s$ takes
$C_G(u)g$ to $C_G(u)gs$.
\end{proof}

\subsection{Houghton groups}\label{sec:houghton}
Put $\mbb N=\{1,2,\ldots\}$ and
$X_n=\{1,\ldots,n\}\times\mbb N$.
\begin{definition}\label{def:houghton}
The \emph{Houghton group} $\mcl G_n$ is the group of all bijections
$\sigma\colon X_n\to X_n$ such that there exist a finite set
$K\subseteq X_n$ and numbers $t_1(\sigma),\ldots,t_n(\sigma)\in\mbb Z$
satisfying
\[
\sigma(i,m)=(i,m+t_i(\sigma))
\]
for all $(i,m)\notin K$. The number $t_i(\sigma)$ is called the shift of
$\sigma$ on the $i$th ray. Since $\sigma$ is a bijection, we have
$\sum_{i=1}^n t_i(\sigma)=0$.
\end{definition}

These groups were constructed by Houghton and used
in~\cite[Example~2.1]{scott1977ends}; later they were described
in~\cite{houghton1978}.

The group $\mcl G_1$ consists of the finitary permutations (i.e.,
permutations with finite support) of $\mbb N$ and is not finitely
generated. For $n\geqslant2$, the group $\mcl G_n$ is finitely generated.
For $1\leqslant i\leqslant n-1$, put
\[
g_i(j,m)=
\begin{cases}
(j,m-1),&j=1,\ m>1,\\
(i+1,1),&j=1,\ m=1,\\
(j,m+1),&j=i+1,\\
(j,m),&\text{otherwise}.
\end{cases}
\]
For $n\geqslant3$, the group $\mcl G_n$ is generated by the set
$S_n=\{g_1,\ldots,g_{n-1}\}$ \cite[Lemma~2.7]{lee2012}. For $n=2$, let
$\beta\colon X_2\to X_2$ be the transposition interchanging $(1,1)$ and
$(2,1)$ and fixing all other points. Then $S_2=\{g_1,\beta\}$ generates
$\mcl G_2$. The group $\mcl G_n$ is of type $\mathrm{FP}_{n-1}$ but not of
type $\mathrm{FP}_n$ \cite{brown1987}. In describing Houghton groups and
their action, we follow the notation of \cite[\S~6]{azu24}; further
details can also be found there.

Consider the right action of $\mcl G_n$ on $X_n$ given by
$x\cdot\sigma=\sigma^{-1}(x)$. Put $x_0=(1,1)$ and
$H_n=\opn{Stab}_{\mcl G_n}(x_0)$. The graph
\[
Y_n=\opn{Sch}(\mcl G_n,H_n,S_n)
\]
has exactly $n$ ends. Indeed, the map $H_n\sigma\mapsto x_0\cdot\sigma$
identifies $H_n\backslash\mcl G_n$ with $X_n$; under this identification,
the generator $g_i$ joins adjacent points of the ray $\{1\}\times\mbb N$,
joins $x_0$ to $(i+1,1)$, joins adjacent points of the ray
$\{i+1\}\times\mbb N$, and gives loops at all other vertices (for $n=2$,
the generator $\beta$ adds the edge $\{x_0,(2,1)\}$ and loops). Removing
the vertex $x_0$, we obtain $n$ components; after removing loops and
identifying parallel edges, each of these components is a ray. Therefore
$e(Y_n)=n$.

Figure~\ref{fig:schreier} shows the cases $n=2$ and $n=4$. For $n=2$, the
vertices $x_0$ and $(2,1)$ are joined by two edges, coming from $g_1$ and
from $\beta$. For $n=4$, adjacent vertices of the first ray are joined by
three parallel edges, and adjacent vertices of each of the other rays are
joined by one edge. Dashed lines indicate the continuation of the rays.
Loops coming from generators that fix vertices are not shown, since they
do not affect the number of ends. Outside the finite region $D$, there
remain $n$ infinite components.

Moreover, $H_n$ is isomorphic to $\mcl G_n$, since the set
$X_n\setminus\{x_0\}$ is also a disjoint union of $n$ rays.
\begin{figure}[ht]
\centering
\tikzset{
  vertex/.style={circle,fill,inner sep=1.3pt},
  edge/.style={thick},
  tail/.style={thick,dashed},
  lbl/.style={font=\footnotesize,inner sep=2pt}
}
\begin{tikzpicture}
\draw[gray,dashed] (0,0) circle (0.5cm);
\node[gray,lbl] at (215:0.78) {$D$};
\node[vertex] (a0) at (0,0) {};
\node[vertex] (a1) at (-1.5,0) {};
\node[vertex] (a2) at (-3,0) {};
\node[vertex] (a3) at (-4.5,0) {};
\node[vertex] (b1) at (2,0) {};
\node[vertex] (b2) at (3.5,0) {};
\node[vertex] (b3) at (5,0) {};
\draw[edge] (a3) -- (a2) -- (a1) -- (a0);
\draw[tail] (a3) -- ++(-1.2,0);
\draw[edge] (a0) to[out=35,in=145] (b1);
\draw[edge] (a0) to[out=-35,in=-145] (b1);
\draw[edge] (b1) -- (b2) -- (b3);
\draw[tail] (b3) -- ++(1.2,0);
\node[lbl,anchor=south east] at ([xshift=-1pt,yshift=1pt]a0) {$x_0$};
\node[lbl,anchor=north west] at ([xshift=1pt,yshift=-1pt]b1) {$(2,1)$};
\node[lbl,anchor=south] at (-2.25,0.05) {$g_1$};
\node[lbl,anchor=south] at (1,0.36) {$g_1$};
\node[lbl,anchor=north] at (1,-0.36) {$\beta$};
\node[lbl,anchor=south] at (2.75,0.05) {$g_1$};
\node[lbl,anchor=north] at (-5.3,-0.12) {$\{1\}\times\mbb N$};
\node[lbl,anchor=north] at (5.8,-0.12) {$\{2\}\times\mbb N$};
\end{tikzpicture}\\[2pt]
{\small (a) $Y_2$}\\[12pt]
\begin{tikzpicture}
\draw[gray,dashed] (0,0) circle (2.15cm);
\node[gray,lbl] at (235:2.45) {$D$};
\node[vertex] (t0) at (0,0) {};
\node[vertex] (t1) at (-1.5,0) {};
\node[vertex] (t2) at (-3,0) {};
\node[vertex] (t3) at (-4.5,0) {};
\draw[edge] (t0) -- (t1);
\draw[edge] (t0) to[out=160,in=20] (t1);
\draw[edge] (t0) to[out=200,in=-20] (t1);
\draw[edge] (t1) -- (t2);
\draw[edge] (t1) to[out=160,in=20] (t2);
\draw[edge] (t1) to[out=200,in=-20] (t2);
\draw[edge] (t2) -- (t3);
\draw[edge] (t2) to[out=160,in=20] (t3);
\draw[edge] (t2) to[out=200,in=-20] (t3);
\draw[tail] (t3) -- ++(-1.2,0);
\node[lbl,anchor=south east] at ([xshift=-3pt,yshift=2pt]t0)
  {$x_0=(1,1)$};
\node[lbl,anchor=south] at (-3.75,0.35) {$g_1,g_2,g_3$};
\node[lbl,anchor=north] at (-5.3,-0.12)
  {$\{1\}\times\mbb N$};
\node[vertex] (r21) at (65:1.5) {};
\node[vertex] (r22) at (65:3.0) {};
\node[vertex] (r23) at (65:4.5) {};
\draw[edge] (t0) -- (r21) -- (r22) -- (r23);
\draw[tail] (r23) -- (65:5.6);
\node[lbl,anchor=east] at (65:0.95) {$g_1$};
\node[lbl,anchor=south west] at (65:5.75)
  {$\{2\}\times\mbb N$};
\node[vertex] (r31) at (22:1.5) {};
\node[vertex] (r32) at (22:3.0) {};
\node[vertex] (r33) at (22:4.5) {};
\draw[edge] (t0) -- (r31) -- (r32) -- (r33);
\draw[tail] (r33) -- (22:5.6);
\node[lbl,anchor=south] at (22:0.95) {$g_2$};
\node[lbl,anchor=west] at (22:5.75)
  {$\{3\}\times\mbb N$};
\node[vertex] (r41) at (-65:1.5) {};
\node[vertex] (r42) at (-65:3.0) {};
\node[vertex] (r43) at (-65:4.5) {};
\draw[edge] (t0) -- (r41) -- (r42) -- (r43);
\draw[tail] (r43) -- (-65:5.6);
\node[lbl,anchor=east] at (-65:0.95) {$g_3$};
\node[lbl,anchor=north west] at (-65:5.75)
  {$\{4\}\times\mbb N$};
\end{tikzpicture}\\[2pt]
{\small (b) $Y_4$}
\caption{The Schreier graphs $Y_2$ and $Y_4$.}
\label{fig:schreier}
\end{figure}

\subsection{Conjugacy diagram with \texorpdfstring{$n$}{n} ends}\label{sec:Wn-ends}
We fix the following definition of wreath products
(see~\cite{Bhattacharjee1998}).
\begin{definition}\label{def:wreath-product}
Suppose a group $Q$ acts on a set $X$ from the right and $A$ is a group.
Denote by $A^{(X)}$ the group of functions $a\colon X\to A$ with finite
support $\opn{supp}(a)=\{x:a(x)\ne1_A\}$ and pointwise multiplication. The
right action of $Q$ on $A^{(X)}$ is given by
\[
(a\cdot\sigma)(x)=a(x\cdot\sigma^{-1}).
\]

The \emph{restricted permutational wreath product} of $A$ and $Q$ with
respect to $X$ is the group
\[
A\wr_XQ:=A^{(X)}\rtimes Q
\]
with multiplication $(a,\sigma)(b,\rho)=((a\cdot\rho)\,b,\sigma\rho)$. The
identity element is $(1,1)$, and
$(a,\sigma)^{-1}=(a^{-1}\cdot\sigma^{-1},\sigma^{-1})$. If $X=Q$ and $Q$
acts on itself by right translations, we obtain the ordinary restricted
wreath product $A\wr Q$.
\end{definition}

If $A$ is written additively, then its identity element is denoted by $0$,
and the multiplication formula takes the form
\begin{equation}\label{eq:wreath-product}
(a,\sigma)(b,\rho)=(a\cdot\rho+b,\sigma\rho).
\end{equation}

Let $n\geqslant2$. Put
\[
A_n=\mbb Z^{(X_n)}=\bigoplus_{x\in X_n}\mbb Z e_x,
\qquad
W_n=\mbb Z\wr_{X_n}\mcl G_n=A_n\rtimes\mcl G_n
\]
and $u=(e_{x_0},1)\in W_n$. In what follows, we identify $A_n$ with the
subgroup $\{(a,1):a\in A_n\}$ and $\mcl G_n$ with the subgroup
$\{(0,\sigma):\sigma\in\mcl G_n\}$.
\begin{proposition}\label{prop:Wn-fg}
The group $W_n$ is generated by the finite set $T_n=S_n\cup\{u\}$.
\end{proposition}

\begin{proof}
The subgroup $\langle T_n\rangle$ contains $\mcl G_n$ and $u$.
By~\eqref{eq:wreath-product}, for $\sigma\in\mcl G_n$ we have
\[
(0,\sigma)^{-1}u(0,\sigma)=(e_{x_0}\cdot\sigma,1)
=(e_{x_0\cdot\sigma},1).
\]
The action of $\mcl G_n$ on $X_n$ is transitive; hence these conjugates
generate $A_n$. Finally, each element of $W_n$ has the form
$(a,\sigma)=(0,\sigma)(a,1)$.
\end{proof}

\begin{proposition}\label{prop:centralizer}
$C_{W_n}(u)=A_n\rtimes H_n$.
\end{proposition}

\begin{proof}
For $(a,\sigma)\in W_n$, a direct computation gives
$u^{(a,\sigma)}=(e_{x_0\cdot\sigma},1)$. Therefore
$u^{(a,\sigma)}=u$ if and only if $\sigma\in H_n$; there is no restriction
on $a$.
\end{proof}

\begin{theorem}\label{thm:Wn-ends}
The component $\opn{conj}(W_n)_{[u]}$ has exactly $n$ ends.
\end{theorem}

\begin{proof}
By Propositions~\ref{prop:conj-schreier} and~\ref{prop:centralizer}, this
component is isomorphic to $\opn{Sch}(W_n,C_{W_n}(u),T_n)$. Consider the
map
\[
\Psi\colon (A_n\rtimes H_n)\backslash W_n
\longrightarrow H_n\backslash\mcl G_n,
\qquad
(A_n\rtimes H_n)(a,\sigma)\longmapsto H_n\sigma.
\]
This map is well defined and bijective. Indeed,
\[
(a,\sigma)(b,\rho)^{-1}=((a-b)\cdot\rho^{-1},\sigma\rho^{-1});
\]
hence the two cosets are equal if and only if $\sigma\rho^{-1}\in H_n$,
i.e., $H_n\sigma=H_n\rho$.
The map $\Psi$ takes the edge with label $s\in S_n$ to the edge
$H_n\sigma\longleftrightarrow H_n\sigma s$. The generator $u$ gives a loop
at each vertex, since
$(a+e_{x_0},\sigma)(a,\sigma)^{-1}=(e_{x_0}\cdot\sigma^{-1},1)\in A_n$.
Thus the graph $\opn{Sch}(W_n,C_{W_n}(u),T_n)$ is obtained from $Y_n$ by
adding loops, and loops do not change the number of ends.
\end{proof}

\begin{corollary}\label{cor:Wn-many-ends}
For $n\geqslant3$, the group $W_n$ is finitely generated and has a
conjugacy class such that the diagram of this class has a finite number of
ends greater than two.
\end{corollary}

Indeed, this follows from Proposition~\ref{prop:Wn-fg} and
Theorem~\ref{thm:Wn-ends}.

\section{Applications to derivations and the Hochschild complex}\label{sec:applications}

In this section, we apply the results of Section~\ref{sec:ends} to
derivations of group algebras and their Hochschild cohomology.

\subsection{Components of the space of derivations}
\label{sec:der}

Throughout the sequel, $G$ is a finitely generated group with a fixed
finite generating set $S$, and $\CG$ denotes the group algebra of $G$.
\begin{definition}\label{def:derivations}
A \emph{derivation} is a linear operator $d\colon\CG\to\CG$ satisfying the
Leibniz rule $d(xy)=d(x)y+xd(y)$. The space of all derivations is denoted
by $\Der(\CG)$. For $c\in\CG$, the operator $d_c\colon x\mapsto[c,x]$ is
called \emph{inner}; the space of inner derivations is denoted by
$\Inn(\CG)$.
\end{definition}

\begin{definition}\label{def:conj-groupoid}
The \emph{groupoid of the adjoint action} $\Gamma$ is defined by
$\opn{Obj}(\Gamma)=G$ and
$\opn{Hom}(\Gamma)=\{(g,h):g,h\in G\}$. A morphism $(g,h)$ has source
$h^{-1}g$ and target $gh^{-1}$, and the composition is given by
\[
(g_2,h_2)\circ(g_1,h_1)=(h_2g_1,h_2h_1).
\]
Morphisms in $\opn{Hom}(x,x)$ are called \emph{loops}; they have the form
$(xz,z)$, where $z\in C_G(x)$.
\end{definition}

This parametrization of morphisms goes back to \cite{AM19}. The assignment
$(g,h)\mapsto(h^{-1}g,h)$ defines an isomorphism of $\Gamma$ onto the
action groupoid $\Gamma_\lambda$ from \cite[Section~1.1]{aru23} for the
conjugation action $\lambda(h,x)=hxh^{-1}$: in \cite{aru23}, the morphism
$(x,h)$ has source $x$ and target $hxh^{-1}$.

We also write the morphism $(g,h)$ as
\[
h^{-1}g\xrightarrow{\ h\ }gh^{-1}.
\]
In other words, the arrow with label $h$ and source $x$ has the form
$x\xrightarrow{h}hxh^{-1}=x^{h^{-1}}$ and corresponds to the pair
$(hx,h)$. If this arrow is followed by an arrow with label $h'$, then the
label of the composition is $h'h$.
Denote by $\Gamma_{[u]}$ the full subgroupoid of $\Gamma$ with object set
$[u]=u^G$.
\begin{remark}\label{rem:conj-groupoid-graph}
If we take the objects of $\Gamma$ as vertices and the morphisms $(g,s)$,
$s\in S$, as edges, forget the orientation of these edges, and identify
edges with the same endpoints and labels, then we obtain the conjugacy
diagram $\opn{conj}(G)$ from Definition~\ref{def:conj-graph}. Under this
correspondence, the component $\Gamma_{[u]}$ corresponds to
$\opn{conj}(G)_{[u]}$.
\end{remark}

\begin{definition}\label{def:local-character}
A \emph{character} on the groupoid $\Gamma$ is a function
$\chi\colon\opn{Hom}(\Gamma)\to\Cx$ additive with respect to composition.
A character is called \emph{locally finite} if for each $h\in G$ the value
$\chi((g,h))$ is nonzero only for finitely many $g\in G$ (cf.
\cite[Definitions~2 and~3]{aru23}).
\end{definition}

\begin{theorem}[\cite{AM19}]\label{thm:char}
The correspondence $d\mapsto\chi_d$ given by
\[
d(h)=h\sum_{y\in G}\chi_d(hy,h)y,
\qquad h\in G,
\]
is an isomorphism between $\Der(\CG)$ and the space of locally finite
characters on $\Gamma$.
\end{theorem}

In terms of the groupoid $\Gamma_\lambda$, this statement is given in
\cite[Section~4.1]{aru23}.

\begin{definition}\label{def:qinn}
A derivation is called \emph{quasi-inner} if the corresponding character
vanishes on all loops. The space of quasi-inner derivations is denoted by
$\QInn(\CG)$.
\end{definition}

Denote by $\Der(\Gamma_{[u]})$ the space of derivations with characters
vanishing outside $\Gamma_{[u]}$. Further, put
\[
\QInn(\Gamma_{[u]})=
\Der(\Gamma_{[u]})\cap\QInn(\CG),
\]
and denote by $\Inn(\Gamma_{[u]})$ the space of derivations $d_c$ with
$c\in\opn{span}_{\Cx}[u]$.
In general, these spaces are not closed under the commutator
(see \cite{arutyunov2023gradingstructurederivationsgroup}).

A character $\chi$ supported on $\Gamma_{[u]}$ and vanishing on loops has
the form $\chi(\varphi)=p(y)-p(x)$ for each arrow $\varphi\colon x\to y$,
where $p\colon[u]\to\Cx$ is a function (a potential) defined up to an
additive constant. Such a character is locally finite if and only if,
outside a finite subgraph, the potential $p$ is constant on each infinite
component \cite[Lemma~2]{aru23}.
Inner derivations correspond to potentials with finite support: for
$c=\sum_{x\in[u]}c_x x$, Theorem~\ref{thm:char} gives
$\chi_{d_c}((hy,h))=c_{hyh^{-1}}-c_y$, i.e., the function $x\mapsto c_x$ is
a potential of $\chi_{d_c}$.

The following result is based on these observations.
\begin{theorem}[{\cite[Theorem~1 and Corollary~2]{aru23}}]\label{thm:qinn}
Suppose the class $[u]$ is infinite and the diagram $\opn{conj}(G)_{[u]}$
has finitely many ends. Then
\[
\dim\bigl(\QInn(\Gamma_{[u]})/\Inn(\Gamma_{[u]})\bigr)
=e\bigl(\opn{conj}(G)_{[u]}\bigr)-1,
\]
where $e(\cdot)$ denotes the number of ends of a graph. In particular, the
right-hand side does not depend on the choice of the finite generating
set $S$.
\end{theorem}

\begin{corollary}\label{cor:anydim}
For each $k\in\mbb Z_{\geqslant0}$, there exist a finitely generated group
$G$ and an infinite conjugacy class $[u]$ such that
\[
\dim\bigl(\QInn(\Gamma_{[u]})/\Inn(\Gamma_{[u]})\bigr)=k.
\]
\end{corollary}

\begin{proof}
For $k\geqslant1$, take $G=W_{k+1}$ and $u=(e_{x_0},1)$. By
Theorem~\ref{thm:Wn-ends}, the corresponding component has $k+1$ ends. For
$k=0$, take $G=\mbb Z^{(\mbb Z^2)}\rtimes\mbb Z^2$, where $\mbb Z^2$ acts by
translations, and $u=(e_0,1)$. This group is generated by $u$ and the
standard generators of $\mbb Z^2$. As in the proof of
Theorem~\ref{thm:Wn-ends}, the diagram of the class $[u]$ is obtained from
the Cayley graph of $\mbb Z^2$ by adding loops; hence this diagram has one
end. In both cases, it remains to apply Theorem~\ref{thm:qinn}.
\end{proof}

\subsection{Admissible characters}\label{subsec:admissible}

The number of ends allows us to compute the dimension of the quotient
space of quasi-inner derivations modulo inner derivations. To describe
outer derivations, we must also deal with characters that are nontrivial on
loops.

For a character $\chi$ on $\Gamma$ and elements $x,g\in G$, put
\[
B_g(x)=\chi((xg,g)),
\]
where $(xg,g)$ is the morphism from $g^{-1}xg$ to $x$.
If $g\in C_G(x)$, then this is the value of the character on a loop
at~$x$.

\begin{definition}\label{def:admissible}
A group character $\kappa\colon C_G(u)\to(\Cx,+)$ is called
\emph{admissible} if there exists a locally finite character $\chi$ on
$\Gamma$ supported on $\Gamma_{[u]}$ such that $\kappa(z)=B_z(u)$ for all
$z\in C_G(u)$. The space of admissible characters is denoted by
$\opn{Adm}_u(G)$.
\end{definition}

Since $\opn{Adm}_u(G)\subseteq\opn{Hom}(C_G(u),\Cx)$ and
$\opn{Hom}(C_G(u),\Cx)\cong
\opn{Hom}_{\mbb Q}\bigl(C_G(u)^{\mathrm{ab}}\otimes\mbb Q,\Cx\bigr)$,
the dimension of $\opn{Adm}_u(G)$ does not exceed the rank of the
abelianization $C_G(u)^{\mathrm{ab}}$ provided this rank is finite.

\begin{proposition}\label{prop:admissible}
Restriction of characters to the loops at the object $u$ defines an
isomorphism
\[
\Der(\Gamma_{[u]})/\QInn(\Gamma_{[u]})
\cong\opn{Adm}_u(G).
\]
In particular, the sequence
\begin{equation}\label{eq:admissible-sequence}
0\longrightarrow
\QInn(\Gamma_{[u]})/\Inn(\Gamma_{[u]})
\longrightarrow
\Der(\Gamma_{[u]})/\Inn(\Gamma_{[u]})
\longrightarrow
\opn{Adm}_u(G)
\longrightarrow0
\end{equation}
is exact. If the class $[u]$ is infinite and the diagram
$\opn{conj}(G)_{[u]}$ has finitely many ends, then
\begin{equation}\label{eq:admissible-dimension}
\dim\bigl(\Der(\Gamma_{[u]})/\Inn(\Gamma_{[u]})\bigr)
=e\bigl(\opn{conj}(G)_{[u]}\bigr)-1+\dim\opn{Adm}_u(G).
\end{equation}
\end{proposition}

\begin{proof}
For $d\in\Der(\Gamma_{[u]})$, put $\chi=\chi_d$ and
$\kappa_d(z)=B_z(u)$, where $z\in C_G(u)$. The additivity of $\chi$ shows
that $\kappa_d$ is a group character: the loop corresponding to $z_1z_2$
is the composition of the loops corresponding to $z_2$ and $z_1$. By
Definition~\ref{def:admissible}, the image of the map $d\mapsto\kappa_d$ is
$\opn{Adm}_u(G)$.
Let $\eta\colon u\to x$, and let $\omega$ be a loop at $x$. Then
\[
\chi(\eta^{-1}\circ\omega\circ\eta)=\chi(\omega),
\]
since the values of $\chi$ on $\eta$ and $\eta^{-1}$ have opposite signs.
Therefore a character vanishes on the loops at $u$ if and only if it
vanishes on all loops of the component. Hence the kernel of the map
$d\mapsto\kappa_d$ is $\QInn(\Gamma_{[u]})$.
Formula~\eqref{eq:admissible-dimension} now follows from
Theorem~\ref{thm:qinn}.
\end{proof}

We note some general properties of loops.
\begin{proposition}\label{prop:restriction-loops}
Let $\chi$ be a locally finite character on $\Gamma_{[u]}$, let $x\in[u]$,
and let $z,z'\in C_G(x)$. Then
\begin{enumerate}
\item $B_{zz'}(x)=B_z(x)+B_{z'}(x)$;
\item $B_{gzg^{-1}}(gxg^{-1})=B_z(x)$ for all $g\in G$;
\item for a fixed $z\in G$, the value $B_z(y)$ is nonzero only for finitely
many $y\in[u]\cap C_G(z)$.
\end{enumerate}
\end{proposition}

\begin{proof}
The loops with labels $z'$ and $z$ can be traversed successively at the
object $x$; this gives the first equality. Let $\eta$ denote the arrow
$x\xrightarrow{g}gxg^{-1}$. Then
$\eta\circ(x\xrightarrow{z}x)\circ\eta^{-1}$
is a loop with label $gzg^{-1}$. The value of $\chi$ on this loop is
$B_z(x)$, since $\chi(\eta)+\chi(\eta^{-1})=0$.
The third assertion follows from the local finiteness of $\chi$ for the
fixed label $z$.
\end{proof}

\subsection{Admissible characters of \texorpdfstring{$W_n$}{Wn}}\label{subsec:Wn-der}
For $x\in X_n$, denote by $u_x$ the element $(e_x,1)$ of $W_n$; in
particular, $u=u_{x_0}$. Then $[u]=\{u_x:x\in X_n\}$, and each arrow of
this component has the form
\[
u_x\xrightarrow{\ (a,\sigma)\ }u_{\sigma(x)},
\qquad a\in A_n,\quad \sigma\in\mcl G_n.
\]
In terms of the right action introduced above,
$x\cdot\sigma^{-1}=(\sigma^{-1})^{-1}(x)=\sigma(x)$.
Recall that $t_j(\sigma)$ is the shift of $\sigma$ on the $j$th ray:
$\sigma(j,m)=(j,m+t_j(\sigma))$ for all sufficiently large $m$.
The shifts of the generator $g_i$, $1\leqslant i\leqslant n-1$, are $-1$
on the first ray and $1$ on the $(i+1)$th ray; $g_i$ fixes all other rays
pointwise.
Denote by $F\trianglelefteq\mcl G_n$ the group of finitary permutations
of $X_n$.
In accordance with the identification of $A_n$ and $\mcl G_n$ with
subgroups of $W_n$, we write $B_a=B_{(a,1)}$ for $a\in A_n$ and
$B_\sigma=B_{(0,\sigma)}$ for $\sigma\in\mcl G_n$.
\begin{lemma}\label{lem:houghton-chars}
Let $\chi$ be a locally finite character on $\Gamma_{[u]}$. Then for some
$\mu\in\Cx$ and all $\sigma\in H_n$, we have
\[
B_\sigma(u)=
\begin{cases}
\mu t_1(\sigma),&n=2,\\
0,&n\geqslant3.
\end{cases}
\]
\end{lemma}

\begin{proof}
Since the elements of $H_n$ fix $x_0$, the map $\sigma\mapsto B_\sigma(u)$
is a group character on $H_n$. This character vanishes on $F\cap H_n$,
since each element of this group has finite order.
Suppose $\sigma\in H_n$ fixes infinitely many points of $X_n$.
By local finiteness, there exists a fixed point $x\ne x_0$ of $\sigma$
such that $B_\sigma(u_x)=0$.
Let $f\in F$ be the transposition of $x$ and $x_0$.
Since $\sigma$ fixes both points, $\sigma$ commutes with $f$.
By Proposition~\ref{prop:restriction-loops}(2), we get
\[
0=B_\sigma(u_x)
=B_{f\sigma f^{-1}}(u_{f(x)})
=B_\sigma(u).
\]
For $n=2$, we have
\[
\ker(t_1|_{H_2})=F\cap H_2
\subseteq\ker\bigl(\sigma\mapsto B_\sigma(u)\bigr).
\]
By the homomorphism theorem, the homomorphism $\sigma\mapsto B_\sigma(u)$
factors through $t_1|_{H_2}$. In addition, $t_1(H_2)=\mbb Z$: indeed,
$\beta g_1\in H_2$ and $t_1(\beta g_1)=-1$. Since every homomorphism
$\mbb Z\to\Cx$ has the form $k\mapsto\mu k$, we obtain
$B_\sigma(u)=\mu t_1(\sigma)$.
For $n\geqslant3$, the kernel of the homomorphism
$\sigma\mapsto(t_1(\sigma),\ldots,t_n(\sigma))$ on $H_n$ is
$F\cap H_n$. The homomorphism $\sigma\mapsto B_\sigma(u)$ vanishes on this
kernel; therefore, by the homomorphism theorem, this homomorphism is
determined by the shift vector.
The shift vectors lie in the group
\[
\bigl\{(k_1,\ldots,k_n)\in\mbb Z^n:
k_1+\cdots+k_n=0\bigr\}.
\]
For each $i=1,\ldots,n-1$, let $\beta_i\in F$ be the transposition of
$x_0$ and $(i+1,1)$. Then $h_i=\beta_i g_i\in H_n$ fixes all points of the
rays other than the first and the $(i+1)$th.
Hence $B_{h_i}(u)=0$.
The shift vectors of $h_i$ coincide with those of $g_i$ and form a basis
of the group above. Therefore the character $\sigma\mapsto B_\sigma(u)$
vanishes on all of $H_n$.
\end{proof}

\begin{theorem}\label{thm:Wn-admissible}
Let $u=(e_{x_0},1)\in W_n$. Then
\[
\opn{Adm}_u(W_n)=
\begin{cases}
\Cx\varepsilon\oplus\Cx\tau,&n=2,\\
\Cx\varepsilon,&n\geqslant3,
\end{cases}
\]
where the characters $\varepsilon$ and $\tau$ on $C_{W_n}(u)=A_n\rtimes H_n$
are given by
\[
\varepsilon(a,\sigma)=a(x_0),
\qquad
\tau(a,\sigma)=t_1(\sigma)\quad(n=2).
\]
Consequently,
\[
\Der(\Gamma_{[u]})/\QInn(\Gamma_{[u]})
\cong
\begin{cases}
\Cx^2,&n=2,\\
\Cx,&n\geqslant3,
\end{cases}
\]
and
\[
\Der(\Gamma_{[u]})/\Inn(\Gamma_{[u]})
\cong
\begin{cases}
\Cx^3,&n=2,\\
\Cx^n,&n\geqslant3.
\end{cases}
\]
The isomorphism in the last formula is not canonical.
\end{theorem}

\begin{proof}
Take an arbitrary locally finite character $\chi$ supported on
$\Gamma_{[u]}$. First consider loops with labels in $A_n$.
Since $A_n$ is abelian and contains $[u]$, each element of $A_n$ commutes
with all elements of $[u]$. Therefore the properties from
Proposition~\ref{prop:restriction-loops} apply to the functions
$B_a(u_x)$.
Conjugating by the arrow with label $(0,f)$, $f\in F$, we get
\[
B_{e_{f(y)}}(u_{f(x)})=B_{e_y}(u_x).
\]
The group $F$ is transitive on the pairs $(x,y)$ with $x=y$ and on the
pairs with $x\ne y$. Hence $B_{e_y}(u_x)$ takes one value for $x=y$ and
another value for $x\ne y$. The second value is zero, since for a fixed
$e_y$ the function $x\mapsto B_{e_y}(u_x)$ has finite support. By
additivity in $a$, this yields $B_a(u_x)=B_{a(x)e_x}(u_x)$.
The value $B_{e_x}(u_x)$ does not depend on $x$; denote this value by
$\theta$. Then $B_a(u_x)=\theta\,a(x)$.
For labels of the form $(0,\sigma)$, we use Lemma~\ref{lem:houghton-chars}.
By Proposition~\ref{prop:centralizer}, any loop at $u=u_{x_0}$ has a label
$(a,\sigma)$, where $a\in A_n$ and $\sigma\in H_n$.
The equality $(a,\sigma)=(0,\sigma)(a,1)$ decomposes this loop into the
loop with label $(a,1)$ followed by the loop with label $(0,\sigma)$.
Therefore the value of $\chi$ on this loop is
\[
B_{(a,\sigma)}(u)=B_a(u)+B_\sigma(u)
=
\begin{cases}
\theta a(x_0)+\mu t_1(\sigma),&n=2,\\
\theta a(x_0),&n\geqslant3.
\end{cases}
\]
This proves that the space of admissible characters is contained in the
space indicated in the statement.
To prove the reverse inclusion, on an arrow
$\varphi\colon u_x\xrightarrow{(a,\sigma)}u_{\sigma(x)}$
put $\chi_\varepsilon(\varphi)=a(x)$.
If the arrow with label $(a,\sigma)$ is followed by an arrow with label
$(b,\rho)$, then the label of the composition is
$(b,\rho)(a,\sigma)=(b\cdot\sigma+a,\rho\sigma)$, and
$(b\cdot\sigma)(x)=b(\sigma(x))$. Hence $\chi_\varepsilon$ is additive.
The character $\chi_\varepsilon$ is locally finite, since $a$ has finite
support. Moreover, if $(a,\sigma)$ is the label of a loop at $u$, then the
value of $\chi_\varepsilon$ on this loop is $a(x_0)=\varepsilon(a,\sigma)$.
For $n=2$, consider the potential $p\colon X_2\to\mbb Z$,
$p(1,m)=-m$, $p(2,m)=m$, and on the same arrow $\varphi$ put
\[
\chi_\tau(\varphi)=p(\sigma(x))-p(x)+t_1(\sigma).
\]
The differences of the values of $p$ add up under composition, and $t_1$
is a group character. For $i=1,2$ and all sufficiently large $m$, we have
$p(\sigma(i,m))-p(i,m)=-t_1(\sigma)$.
Hence $\chi_\tau$ is locally finite. If $\sigma\in H_2$, then
$\sigma(x_0)=x_0$, and the value of $\chi_\tau$ on the loop at $u$ is
$\tau(a,\sigma)$.
Since $t_1(H_2)=\mbb Z$ by the proof of Lemma~\ref{lem:houghton-chars},
the characters $\varepsilon$ and $\tau$ are linearly independent.
Extending both characters by zero outside $\Gamma_{[u]}$, we obtain all
the asserted admissible characters. The dimensions of the quotients follow
from Proposition~\ref{prop:admissible} and
Theorems~\ref{thm:qinn} and~\ref{thm:Wn-ends}.
\end{proof}

\begin{remark}\label{rem:Wn-Hom}
For $W_n$, the inclusion
$\opn{Adm}_u(W_n)\subseteq\opn{Hom}(C_{W_n}(u),\Cx)$ is strict. A character
$\kappa$ of $C_{W_n}(u)=A_n\rtimes H_n$ is determined by the pair
$(\kappa|_{A_n},\kappa|_{H_n})$, where $\kappa|_{A_n}$ is invariant under
$H_n$, and any such pair defines a character. The group $H_n$ has two
orbits on $X_n$, namely $\{x_0\}$ and $X_n\setminus\{x_0\}$; therefore the
$H_n$-invariant characters of $A_n$ have the form
$a\mapsto\lambda_0a(x_0)+\lambda_1\sum_{x\ne x_0}a(x)$. The characters of
$H_n$ factor through the shifts (see the proof of
Lemma~\ref{lem:houghton-chars}) and form a space of dimension $n-1$. Hence
$\dim\opn{Hom}(C_{W_n}(u),\Cx)=n+1$, whereas by
Theorem~\ref{thm:Wn-admissible} the dimension of $\opn{Adm}_u(W_n)$ is $2$
for $n=2$ and $1$ for $n\geqslant3$. The character
$a\mapsto\sum_{x\ne x_0}a(x)$ and, for $n\geqslant3$, all nonzero
characters factoring through the shifts are not admissible. In the
notation of Remark~\ref{rem:LES}, the rank of the map $\partial$ is $1$
for $n=2$ and $n$ for $n\geqslant3$.
\end{remark}

\subsection{Hochschild cohomology by components}\label{sec:hochschild}
For a description of derivations of group algebras in terms of Hochschild
cohomology, see \cite{mish21}. The $0$-cochains of the Hochschild complex
with coefficients in $\CG$ are the elements of $\CG$. For a $0$-cochain
$c\in\CG$ and a $1$-cochain $c'\colon\CG\to\CG$, the differential $\delta$
is given by
\[
(\delta c)(x)=xc-cx,\qquad
(\delta c')(x,y)=xc'(y)-c'(xy)+c'(x)y.
\]
Thus the $1$-cocycles are derivations, and the coboundaries of
$0$-cochains form the space of inner derivations. Consequently,
\[
\opn{HH}^1(\CG,\CG)\cong\Der(\CG)/\Inn(\CG).
\]
For a class $[u]$, put
$C_{[u]}^0=\opn{span}_{\Cx}[u]$. For $k\geqslant1$, denote by
$C_{[u]}^k$ the space of $k$-cochains $c$ such that
\[
c(x_1,\ldots,x_k)(x_1\cdots x_k)^{-1}
\in\opn{span}_{\Cx}[u]
\qquad\text{for all }x_1,\ldots,x_k\in G.
\]
The Hochschild differential preserves this condition, since conjugation
preserves $[u]$. The degree-$1$ cohomology group of the resulting
subcomplex is denoted by $\opn{HH}_{[u]}^1(\CG,\CG)$.
For a derivation $d$, the condition $d(g)g^{-1}\in\opn{span}_{\Cx}[u]$ for
all $g\in G$ is equivalent to the condition that the character of $d$ is
supported on $\Gamma_{[u]}$. Up to sign, the coboundaries of elements of
$C_{[u]}^0$ coincide with the derivations in $\Inn(\Gamma_{[u]})$.
Therefore
\begin{equation}\label{eq:HH-component}
\opn{HH}_{[u]}^1(\CG,\CG)
\cong
\Der(\Gamma_{[u]})/\Inn(\Gamma_{[u]}),
\end{equation}
and sequence~\eqref{eq:admissible-sequence} and
formula~\eqref{eq:admissible-dimension} describe the space
$\opn{HH}_{[u]}^1(\CG,\CG)$.

\begin{remark}\label{rem:LES}
The correspondence $d\mapsto\bigl(g\mapsto d(g)g^{-1}\bigr)$ identifies
$\Der(\Gamma_{[u]})$ with the space of crossed homomorphisms
$G\to\opn{span}_{\Cx}[u]$, where $G$ acts on $\opn{span}_{\Cx}[u]$ by
conjugation, and identifies $\Inn(\Gamma_{[u]})$ with the space of
principal crossed homomorphisms. Therefore
\[
\opn{HH}_{[u]}^1(\CG,\CG)\cong H^1\bigl(G,\opn{span}_{\Cx}[u]\bigr),
\]
where $\opn{span}_{\Cx}[u]\cong\Cx[C_G(u)\backslash G]$ is the
permutation module with the left action of $G$ given by
\[
g\cdot(C_G(u)h)=C_G(u)hg^{-1}.
\]
Denote by $\Cx^{[u]}$ the module of all functions on $[u]$. This module is
coinduced from the trivial $C_G(u)$-module $\Cx$; therefore, by Shapiro's
lemma \cite[Proposition~III.6.2]{brown1982},
\[
H^1\bigl(G,\Cx^{[u]}\bigr)\cong H^1\bigl(C_G(u),\Cx\bigr)
=\opn{Hom}\bigl(C_G(u),\Cx\bigr).
\]
Suppose the class $[u]$ is infinite. The long exact sequence associated
with
\[
0\longrightarrow\opn{span}_{\Cx}[u]\longrightarrow\Cx^{[u]}
\longrightarrow\Cx^{[u]}/\opn{span}_{\Cx}[u]\longrightarrow0,
\]
together with this isomorphism, yields the exact sequence
\begin{multline*}
0\to\Cx\to H^0\bigl(G,\Cx^{[u]}/\opn{span}_{\Cx}[u]\bigr)
\to H^1\bigl(G,\opn{span}_{\Cx}[u]\bigr)\\
\to\opn{Hom}(C_G(u),\Cx)
\xrightarrow{\ \partial\ }H^1\bigl(G,\Cx^{[u]}/\opn{span}_{\Cx}[u]\bigr).
\end{multline*}
The elements of $H^0\bigl(G,\Cx^{[u]}/\opn{span}_{\Cx}[u]\bigr)$ are the
functions on $[u]$ almost invariant under conjugation, considered modulo
finitely supported functions; if the number of ends is finite, then the
dimension of this space equals $e(\opn{conj}(G)_{[u]})$ (cf. the proof of
\cite[Theorem~1]{aru23} and the description of the number of ends of a
pair in terms of almost invariant sets in~\cite{scott1977ends}). The map
to $\opn{Hom}(C_G(u),\Cx)$ coincides with the restriction of characters to
the loops at $u$. Therefore sequence~\eqref{eq:admissible-sequence} is a
fragment of this long exact sequence, $\opn{Adm}_u(G)=\ker\partial$, and
the statement of Theorem~\ref{thm:qinn} is equivalent to the equality
$\dim H^0\bigl(G,\Cx^{[u]}/\opn{span}_{\Cx}[u]\bigr)=e(\opn{conj}(G)_{[u]})$.
\end{remark}

If $G$ is finitely generated, then each derivation involves only finitely
many components, since a derivation is determined by its values on a
finite generating set. The projections of a character to the components
again define derivations, and the projections of an inner derivation are
inner. Therefore
\begin{equation}\label{eq:HH-decomposition}
\opn{HH}^1(\CG,\CG)
\cong\bigoplus_{[v]}
\opn{HH}_{[v]}^1(\CG,\CG),
\end{equation}
where the sum is taken over the conjugacy classes of $G$.
\begin{corollary}\label{cor:Wn-HH-component}
For $u=(e_{x_0},1)\in W_n$, we have
\[
\dim_{\Cx}
\opn{HH}_{[u]}^1\bigl(\Cx[W_n],\Cx[W_n]\bigr)
=
\begin{cases}
3,&n=2,\\
n,&n\geqslant3.
\end{cases}
\]
Moreover, $\opn{HH}^1\bigl(\Cx[W_n],\Cx[W_n]\bigr)$ is
infinite-dimensional for each $n\geqslant2$.
\end{corollary}

\begin{proof}
By Theorem~\ref{thm:Wn-ends}, the diagram of the class $[u]$ has $n$ ends.
By Theorem~\ref{thm:Wn-admissible}, the dimension of $\opn{Adm}_u(W_n)$ is
$2$ for $n=2$ and $1$ for $n\geqslant3$. The first assertion follows from
formulas~\eqref{eq:HH-component} and~\eqref{eq:admissible-dimension}.
The classes $[u^r]$, $r=1,2,\ldots$, are pairwise distinct.
For each $r$, we have
$[u^r]=\{(r e_x,1):x\in X_n\}$ and
$C_{W_n}(u^r)=C_{W_n}(u)$.
By Proposition~\ref{prop:conj-schreier}, the diagram of each class $[u^r]$
is isomorphic to the diagram of the class $[u]$; hence, by
Theorem~\ref{thm:Wn-ends}, this diagram has $n$ ends. By
Theorem~\ref{thm:qinn}, the space
$\QInn(\Gamma_{[u^r]})/\Inn(\Gamma_{[u^r]})$
has dimension $n-1>0$. Therefore infinitely many components make a nonzero
contribution to direct sum~\eqref{eq:HH-decomposition}.
\end{proof}

\printbibliography[
heading=bibintoc,
title={References}
]
\end{document}